\documentclass[10pt,reqno]{amsart}
\usepackage{hyperref}
\usepackage{latexsym,amsmath}
\usepackage{enumerate}
\usepackage{amsfonts}
\usepackage{amssymb}
\usepackage{geometry}
\usepackage{latexsym}
\usepackage{fixmath}
\usepackage[T1]{fontenc}
\usepackage{fourier}
\usepackage{bbm}
\usepackage{color}
\usepackage{xcolor}
\usepackage{fullpage}
\usepackage{cleveref}
\usepackage{ulem}

\newcommand{\sign}{\mathrm{sign}}
\newcommand{\disteq}{\stacksign{d}=}

\renewcommand{\epsilon}{\eps}

\newcommand{\vd}{\vec d}

\newcommand\MU{\vec\mu}

\newcommand\vX{\vec X}

\newcommand\vY{\vec Y}
\newcommand\vT{\vec T}

\newcommand\vm{\vec m}

\newcommand\einf{\vec \vartheta_{o,\mathrm{inf}}}
\newcommand\efin{\vec\vartheta_{o,\mathrm{fin}}}

\newcommand{\LDE}{\mathrm{LL}}

\newcommand\PHI{\vec\Phi}

\newcommand\dd{{\mathrm d}}

\numberwithin{equation}{section}

\renewcommand{\vec}[1]{\boldsymbol{#1}}

\newcommand\SIGMA{\vec\sigma}

\newtheorem{definition}{Definition}[section]

\newtheorem{remark}[definition]{Remark}
\newtheorem{theorem}[definition]{Theorem}
\newtheorem{lemma}[definition]{Lemma}
\newtheorem{proposition}[definition]{Proposition}
\newtheorem{corollary}[definition]{Corollary}

\newtheorem{fact}[definition]{Fact}

\newcommand\cE{\mathcal{E}}

\newcommand\cL{\mathcal{L}}

\newcommand\cP{\mathcal{P}}

\newcommand\cW{\mathcal{W}}

\def\cE{{\mathcal E}}

\newcommand\va{\vec a}
\newcommand\vs{\vec s}

\newcommand\eul{\mathrm{e}}
\newcommand\eps{\varepsilon}

\newcommand\NN{\mathbb{N}}

\newcommand\Erw{\mathbb{E}}
\newcommand{\vecone}{\vec{1}}

\newcommand{\Po}{{\rm Po}}

\newcommand{\BP}{\mathrm{BP}}
\newcommand{\DE}{\BP}

\newcommand\bc[1]{\left({#1}\right)}
\newcommand\cbc[1]{\left\{{#1}\right\}}

\newcommand\brk[1]{\left\lbrack{#1}\right\rbrack}

\newcommand\abs[1]{\left|{#1}\right|}

\newcommand\RR{\mathbb{R}}

\newcommand{\stacksign}[2]{{\stackrel{\mbox{\scriptsize #1}}{#2}}}

\newcommand{\Bollobas}{Bollob\'as}

\newcommand{\Chvatal}{Chv\'{a}tal}

\newcommand\pr{\mathbb{P}} 
\renewcommand\Pr{\pr}

\newcommand\Chap{Chapter}

\newcommand\RSA{Random Structures and Algorithms}

\newcommand\supp{\mathrm{supp}}
\renewcommand\ln{\log}

\newcommand\zhu[1]{{#1}}
\usepackage{tikz}
\usetikzlibrary{arrows}
\usepackage{float}
\usetikzlibrary{snakes}
\tikzset{
  treenode/.style = {align=center, inner sep=0pt, text centered,
    font=\sffamily},
  arn_n/.style = {treenode, circle,thick,draw=black!75,
  			  fill=black!20,minimum size=5mm},% arbre rouge noir, noeud noir
  arn_r/.style = {treenode, rectangle,thick,draw=blue!75,fill=blue!20,minimum size=6mm},% arbre rouge noir, noeud rouge
  arn_x/.style = {treenode,thick,draw=blue!75,minimum size=5mm}% arbre rouge noir, nil
}

\title{Random $2$-SAT: The set of atoms of the limiting empirical marginal distribution\textit{}}
\author[N. Müller \and R. Neininger \and H. Zhu]{Noela Müller, Ralph Neininger, Haodong Zhu}
\thanks{N. M. is supported by the NWO Gravitation project NETWORKS under grant no. 024.002.003. 
H. Z. is supported by the
NWO Gravitation project NETWORKS under grant no. 024.002.003 and the European
Union’s Horizon 2020 research and innovation programme under the Marie
Skłodowska-Curie grant agreement no. 945045. We also thank Amin Coja-Oghlan for insightful discussions and feedback on the topic, as well as and Rajat Subhra Hazra for an inspiring course on the spectra of random matrices during the NETWORKS training week in spring 2024.}

\address{Noela M\"uller, {\tt n.s.muller@tue.nl}, Eindhoven University of Technology, Department of Mathematics and Computer Science, MetaForum MF 4.084, 5600 MB Eindhoven, the Netherlands.}
\address{Ralph Neininger, {\tt neininger@math.uni-frankfurt.de}, Goethe University Frankfurt, Institute of Mathematics, 60629 Frankfurt a.M., Germany }
\address{Haodong Zhu, {\tt h.zhu1@tue.nl}, Eindhoven University of Technology, Department of Mathematics and Computer Science, MetaForum MF 4.084, 5600 MB Eindhoven, the Netherlands.}

\begin{document}

\begin{abstract}
We show that the set of atoms of the limiting empirical marginal distribution in the random $2$-SAT model is $\mathbb Q \cap (0,1)$, for all clause-to-variable densities up to the satisfiability threshold. While for densities up to $1/2$, the measure is purely discrete, we additionally establish the existence of a nontrivial continuous part for any density in $(1/2, 1)$. Our proof is based on the construction of a random variable with the correct distribution as the the root marginal of a multi-type Galton-Watson tree, along with a subsequent analysis of the resulting almost sure recursion. 
\end{abstract}
\maketitle

\section{Introduction and main result}
\subsection{Introduction} Random constraint satisfaction problems provide a flourishing research ground at the crossroads of different disciplines such as mathematics, physics and computer science. One of their key representatives is random $k$-SAT, a model well-known for its notorious analytic hardness. 

The most accessible member of the random $k$-SAT family is random $2$-SAT, which stands out through its early success history: In 1992, it was the first random constraint satisfaction problem whose satisfiability threshold was established rigorously \cite{CR, Goerdt}, an achievement that was followed by even finer analyses of its scaling window \cite{BBCKW, Dovgal}. In contrast, the existence and location of the satisfiability threshold of random $k$-SAT for $k \geq k_0$, where $k_0$ is a large absolute constant, have only been established comparatively recently in the breakthrough paper \cite{DSS3}, while a corresponding result for smaller values of $k$ is still not available. In a certain sense, the early success on random $2$-SAT reflects the fact that $2$-SAT is in P, while $k$-SAT for $k \geq 3$ is (presumably) not. 

Turning to counting problems, the distinction between $k=2$ and $k \geq 3$ through the lens of computational hardness is not as clearly cut: Determining the number of solutions is a \#P-complete task, even restricted to $2$-SAT. 
Indeed, the timeline of the subsequent results closely matches those considerations on complexity: 
After first results \cite{AM, MS}, the typical number of solutions of random $2$-SAT formulae throughout the satisfiable regime was established in 2021 \cite{2sat} (and later refined in \cite{2satclt}), roughly 40 years after the identification of the satisfiability threshold, and even after the seminal work of \cite{DSS3}.

A hallmark of the increased complexity can be found in the statements of the respective results: While the satisfiability threshold in random $2$-SAT occurs at the point where the number of clauses asymptotically exceeds the number of variables, the formula for the asymptotic number of solutions from \cite{2sat} is considerably more involved. Indeed, the proofs in the line of work \cite{AM, 2sat, 2satclt, DSS3, MS}, as well as many other articles, are based on deep ideas from statistical physics. As a consequence, the formula for the exponential growth rate of the number of solutions from \cite{2sat}, which confirms corresponding predictions from statistical physics \cite{MZ}, involves the evaluation of a functional at a probability measure $\pi_d$ on the unit interval, where $d\in(0,2)$ encodes the average number of clauses in which each variable appears. However, $\pi_d$ is only implicitly characterised by a stochastic recursion (the so-called \textit{density evolution}). 
What is more, 
the associated recursion is of a form that is not directly approachable by available results on stochastic recursions of an affine linear or max-type, for example. 

The main purpose of this note is to provide a first step towards a better understanding of the distribution $\pi_d$, which one may think of as the limit of the probability that a uniformly chosen variable within a uniformly random solution is set to `true'. As we will see, for some values of $d$, $\pi_d$ is mixed atomic and continuous. This diverse structure arises from an inhomogeneity among variable marginals which is absent in other random constraint satisfaction problems, such as random $k$-NAESAT or random graph $q$-coloring, where permutation symmetry ensures that all variable marginals are uniform. In random $k$-XORSAT and random linear equations, the algebraic structure ensures that the variable marginals are either uniform or Dirac measures \cite[Lemma 2.3]{jlin}. Even in random $k$-SAT for large values of $k$, the differences between variable marginals are less pronounced because of stronger concentration properties \cite{CMR, DSS3}. 

Simply speaking, the complex marginal structure in random $k$-SAT is owed to the fact that the variables are highly sensitive to imbalances in their local neighbourhood: A variable that appears in many more clauses non-negated than negated will tend to be set to `true' in a significant portion of the solutions, giving rise to an effect suggestively called \textit{populism} in \cite{yuval}. As a consequence of the neighborhood fluctuations between different variables, a rather intricate distribution over different marginals arises upon choosing a variable uniformly at random. 
The inhomogeneous marginal landscape in random $k$-SAT is also closely related to non-concentration effects that preclude an analysis via the very same techniques that have been successfully employed in the study of other, `symmetric' models such as random $k$-NAESAT (see e.g. the discussion in \cite[Sections 1.1 and 3.3]{nae} or \cite{CEJKK, CKM, CKPZ, SSZ}), random graph $q$-coloring (e.g. \cite{CKM}) or random systems of linear equations \cite{jlin}.

As random $2$-SAT is analytically far more accessible and better understood than random $k$-SAT for $k \geq 3$, a detailed study of its marginal structure in terms of the effect of neighbourhood fluctuations suggests itself. We hope that a more detailed understanding of $\pi_d$ and its properties will also shed light on the question what kind of marginal structures one can observe in random constraint satisfaction problems with neighbourhood fluctuations, and  are not aware of results of a similar kind for other models.

\subsection{Main result}  
For any integer $n>1$ and $d>0$, let $\PHI=\PHI_n$ be a random $2$-SAT formula on Boolean variables $x_1,\ldots,x_n$ with a Poisson number $\vm$ of clauses, where $\vm$ follows a Poisson distribution with parameter $dn/2$. Given $\vm$, each of the $\vm$ clauses is drawn independently and uniformly from the set of all $4\binom{n}{2}$ possible clauses with two distinct variables.
The chosen parametrisation ensures that on average, each variable appears in $d$ clauses. 

\Chvatal~and Reed \cite{CR} and Goerdt \cite{Goerdt} proved independently that the model exhibits a sharp threshold in satisfiability at $d_{\mathrm{sat}}=2$: For $d<2$, whp $\PHI$ is satisfiable, while it is not for $d>2$.  
To introduce the notion of a uniformly chosen solution, we denote by $S(\PHI)$ the set of all satisfying assignments of $\PHI$.
Assuming that $S(\PHI)\neq\emptyset$, let
\begin{align}\label{eqGibbs}
\mu_{\PHI}(\sigma)&=\vecone\cbc{\sigma\in S(\PHI)}/Z(\PHI), \qquad \sigma\in\{\pm1\}^{\{x_1,\ldots,x_n\}},
\end{align}
denote the uniform distribution on $S(\PHI)$, where $Z(\PHI) = |S(\PHI)|$. In the above, we encode `true' by $+1$ and `false' by $-1$.
In the following, samples from $\mu_{\PHI}$ are denoted by the boldface notation $\SIGMA$. As indicated earlier, an object of primal interest in the study of random $2$-SAT is the empirical marginal distribution of $\mu_{\PHI}$. In the course of deriving an expression for the asymptotic exponential growth rate of the number of solutions, the asymptotic behaviour of the empirical distribution of the random variables $\mu_{\PHI}(\SIGMA_{x_1}=1), \ldots, \mu_{\PHI}(\SIGMA_{x_n}=1)$ was determined in  \cite{2sat}:

\begin{proposition}[{\cite[Corollary 1.3]{2sat}}]\label{Cor_BP}
For any $0<d<2$ there exists a probability distribution $\pi_d$ on $[0,1]$ such that the random probability measure
\begin{align}\label{eqempirical}
\pi_{\PHI}&=\frac1n\sum_{i=1}^n\delta_{\mu_{\PHI}(\SIGMA_{x_i}=1)}
\end{align}
converges to $\pi_d$ weakly in probability\footnote{Specifically, for any continuous function $f:[0,1]\to\RR$, 
 $\lim_{n\to\infty}\Erw\abs{\int_0^1f(z)\dd\pi_d(z)-\int_0^1f(z)\dd\pi_{\PHI}(z)}=0$.}.
\end{proposition}  
The principal aim of this article is the study of $\pi_d$. More specifically, and omitting the dependence on $d$, let
\[\mathbb{A}_{\mathrm{p.p.}} := \{x \in [0,1]: \pi_d(\{x\})>0\} \] 
denote the pure point support of $\pi_d$. Indeed, the derivation of $\pi_d$ in \cite{2sat} implies that $\mathbb{A}_{\mathrm{p.p.}}$ is a subset of $(0,1)$: 
\begin{remark}[{\cite[Proposition 2.1]{2sat}}]\label{rem:boundary}
    For any $d \in (0,2)$, we have $\pi_d(\{0\}) = \pi_d(\{1\}) = 0$. Thus, $\mathbb{A}_{\mathrm{p.p.}}\subseteq (0,1)$.
\end{remark}

While our primary interest is the pure point support of $\pi_d$, we also determine the support of its continuous part below. Let $\nu$ be a measure on $[0,1]$,  
and let $B_\varepsilon(x)$ for $x\in\RR, \varepsilon>0$,  denote the interval $(x-\varepsilon,x+\varepsilon)$. We then set
\begin{align}\label{rep_supp}
\mathrm{supp}(\nu):=\{x\in[0,1]\,:\, \nu(B_\varepsilon(x)\cap [0,1])>0 \mbox{ for all } \varepsilon>0\}.
\end{align}

Our main result then precisely identifies the set of atoms of $\pi_d$ throughout the satisfiable regime, as well as the support of its continuous part:

\begin{theorem}\label{prop_atoms}
    For any $d \in (0,2)$, the pure point support of $\pi_d$ is
\begin{align}
\mathbb{A}_{\mathrm{p.p.}} = \mathbb Q \cap (0,1).
\end{align}
Moreover, for $d\in (0,1]$, $\pi_d$ is a discrete measure; for $d \in (1,2)$, $\pi_d$ has a non-trivial continuous part $\pi_{d,\mathrm{c}}$ with $\supp(\pi_{d,\mathrm{c}}) = [0,1]$.
\end{theorem}

Looking at finite satisfiable $2$-SAT formulas, the rational numbers in $(0,1)$ constitute a natural candidate set for the marginals of single variables, as these are defined as \textit{proportions} of solutions where the given variable is set to `true' or `false'. However, it might be less immediate that, irrespective of $d$, the pure point support of $\pi_d$ contains \textit{all} rational numbers in $(0,1)$, and that a non-trivial continuous part $\pi_{d,\mathrm{c}}$ exists for $d \in (1,2)$. As our proof shows, non-triviality of $\pi_{d,\mathrm{c}}$ is clearly linked to the approximation of the local neighbourhood structure of a uniformly chosen variable by a supercritical branching process.

\subsection{Proof strategy}\label{sec:proof_strat}
Theorem \ref{prop_atoms} splits naturally into the following two inclusions:

\begin{proposition}\label{lem_atoms1}
    For any $d \in (0,2)$, we have
\begin{align}
\mathbb{A}_{\mathrm{p.p.}} \supseteq \mathbb Q \cap (0,1).
\end{align}
\end{proposition}

\begin{proposition}\label{lem_atoms2}
    For any $d \in (0,2)$, we have
\begin{align} \label{incl:atoms2}
\mathbb{A}_{\mathrm{p.p.}} \subseteq \mathbb Q \cap (0,1).
\end{align}
Moreover, for $d\in (0,1]$, $\pi_d$ is a discrete measure; for $d \in (1,2)$, $\pi_d$ has a non-trivial continuous part $\pi_{d,\mathrm{c}}$ with $\supp(\pi_{d,\mathrm{c}}) = [0,1]$.
\end{proposition}

The proofs of both Propositions~\ref{lem_atoms1} and~\ref{lem_atoms2} rely on the explicit construction of a random variable $\MU_o$ with distribution $\pi_d$, which is described in Section~\ref{sec-as-conv-root-marignal}. To motivate the construction of $\MU_o$,  
recall that Proposition~\ref{Cor_BP} shows that the probability measure $\pi_d$ approximates the marginal distribution of a uniformly chosen variable in a random large formula $\PHI$. Its proof in \cite{2sat} relies on an exploitation of the local neighbourhood structure of such a uniformly chosen variable, in combination with a demonstration of the absence of long-range correlations.
To formally speak about variable neighbourhoods, we represent a given $2$-SAT formula $\Phi$ by its \textit{factor graph} $G(\Phi)$, which can be regarded as a bipartite graph with edge-labels in $\{\pm 1\}$. In the graphical representation $G(\Phi)$, one vertex set $V_n=\{x_1,\ldots,x_n\}$ encodes the propositional variables of $\Phi$, while the other vertex set $F_{m}=\{a_1,\ldots,a_{m}\}$ encodes its clauses. An edge between a variable and a clause vertex is present if and only if the corresponding variable participates in the corresponding clause. Finally, encoding variable negations by $-1$ and non-negations by $+1$, each edge is labelled to indicate the (non-)negation of the incident variable in the incident clause.

Conveniently, in a random formula $\PHI$, the local neighbourhood structure of a uniformly chosen variable can be approximated by that of a multi-type Galton-Watson tree $\vT$ (for the precise definition of $\vT$, see Section~\ref{sec-as-conv-root-marignal}). Therefore, the probability measure $\pi_d$ is defined as the distribution of the root marginal of the random $2$-SAT formula encoded by such a tree in \cite{2sat}. While for the purposes of \cite{2sat}, an ensuing \textit{distributional} characterisation of $\pi_d$ was sufficient, in Proposition~\ref{Prop_asconvergence} below, we show that the sequence of root marginals of the truncated multi-type Galton-Watson tree converges  \textit{almost surely} to a random variable $\MU_o$ with distribution $\pi_d$. This establishes a natural link between $\pi_d$ and the structure of $\vT$, along with an almost sure decomposition of $\MU_o$ that we can capitalize upon.

As a consequence of the construction in Proposition~\ref{Prop_asconvergence}, the proofs of both Propositions~\ref{lem_atoms1} and~\ref{lem_atoms2} become simple and gain a clear interpretation in terms of the local structure around a typical variable. First, the proof of Proposition~\ref{lem_atoms1}, which is given in Section~\ref{sec:lem_atoms1}, is based on the analysis of the distribution of $\MU_o$ on the event that the underlying Galton-Watson tree $\vT$ dies out. This event has positive probability for all $d \in (0,2)$. We show that for any rational number $q \in (0,1)$, there is a finite factor graph $G(\Phi)$ along with a distinguished variable vertex $x$ such that $G(\Phi)$ is a tree and the marginal of $\Phi$ in $x$ is $q$. As any such tree factor graph has positive probability under the law of the Galton-Watson tree underlying the construction of $\MU_o$, Proposition~\ref{lem_atoms1} is an immediate consequence of this statement.

In contrast, the proof of Proposition~\ref{lem_atoms2} in Section~\ref{sec:lem_atoms2} is based on the analysis of the distribution of $\MU_o$ on the event that the underlying Galton-Watson tree $\vT$ is infinite, which has positive probability for $d \in (1,2)$.  Exploiting the recursive nature of the tree and the almost sure decomposition of $\MU_o$, we show that in this case, the conditional distribution of $\MU_o$ cannot have point masses. The proof of this result is based on a `smoothing' argument.

\section{Preliminaries}
The following subsections provide a collection of the results and concepts from \cite{2sat} that this article builds upon. To facilitate reading and referencing to results from \cite{2sat}, we have aligned our notation with the one from the mentioned article. 
We assume throughout that $d \in (0,2)$ and define the continuous and mutually inverse functions
\begin{align}\label{eqll}
\psi:\RR&\to(0,1),\quad z\mapsto\bc{1+\tanh(z/2)}/2,&
\varphi:(0,1)&\to\RR,\quad p\mapsto\log(p/(1-p)).
\end{align}
For $k \in \NN$, we typically use the abbreviation $[k]:=\{1,\ldots, k\}$.

\subsection{{Recursive computation of tree marginals}} \label{subsec:BP_on_trees}

For the characterisation of the root marginals of finite tree formulas in Subsection \ref{subsec:finite_trees}, the following recursive computation of marginals in trees, which also forms the basis of the belief propagation algorithm, will be instrumental.

Recall the graphical representation $G(\Phi)$ of a $2$-SAT formula $\Phi$ described in Section~\ref{sec:proof_strat}, and that for satisfiable formulas $\Phi$ and variables $x$, $\mu_{G(\Phi)}(\SIGMA_x = \cdot)$ denotes the marginal distribution of the uniform distribution over solutions to $\Phi$ in $x$. Throughout this section, assume that $\Phi$ is such that $G(\Phi)$ is a finite connected graph without cycles, which we will refer to as a tree factor graph. In particular, $S(\Phi) \not= \emptyset$. To write down the desired recursions for the marginals of $\Phi$, we next define certain subgraphs (or sub-formulas) that are obtained from $G(\Phi)$ through the deletion of edges: For an edge $\{x,a\}$ in $G(\Phi)$ such that $x$ is a variable node and $a$ is a factor node, let $G(\Phi_{x \to a})$ be the subgraph of $G(\Phi)$ that consists of node $x$ along with all variable and factor nodes that can be reached from $x$ by a self-avoiding path that does not use the edge $\{x,a\}$. In other words, $G(\Phi_{x \to a})$ is the connected component that contains $x$, after removal of the edge $\{x,a\}$. As $G(\Phi)$, $G(\Phi_{x \to a})$ is again the graphical representation of a $2$-SAT formula.

It then turns out that the marginal of $x$ in $G(\Phi)$ can be expressed in terms of the marginals that are at graph-distance $2$ from $x$: Let $\partial x$ denote the collection of factor nodes that are adjacent to $x$ in $G(\Phi)$. Moreover, recall that $G(\Phi)$ is an edge-labelled graph. To make the dependence on the adjacent variables and clauses explicit, for any edge $\{x,a\}$ in $G(\Phi)$, let $\mathrm{s}(x,a):=\sign(x,a) \in \{-1,1\}$ denote the sign with which variable $x$ appears in clause $a$. 
Finally, fixing $x$, for $a \in \partial x$, let $y_a$ denote the unique second variable that appears in clause $a$. It is then well-known and can be proven by induction that the marginals of the sub-tree formulas are related by the following recursions:
\begin{align}\label{eqBPrec}
& \mu_{G(\Phi)}(\SIGMA_x = 1) \nonumber \\
   &=\frac{\prod_{a\in\partial x} (1 - \vecone\{\mathrm{s}(x,a)=-1\}\mu_{G(\Phi_{y_a \to a})}(\SIGMA_{y_a} = -\mathrm{s}(y_a,a)))} 
 	{\prod_{a\in\partial x}(1 - \vecone\{\mathrm{s}(x,a)=-1\}\mu_{G(\Phi_{y_a \to a})}(\SIGMA_{y_a} = -\mathrm{s}(y_a,a))) +\prod_{a\in\partial x}(1 - \vecone\{\mathrm{s}(x,a)=1\}\mu_{G(\Phi_{y_a \to a})}(\SIGMA_{y_a} = -\mathrm{s}(y_a,a)))}.
\end{align}
Comprehensible in-depth derivations of the general form of \eqref{eqBPrec}, along with excellent expositions of belief propagation, can be found in~\cite[\Chap~14]{MM} and \cite[Lecture 1]{Bolthausen}.

\subsection{{Distributional identities from \cite{2sat}}}
We next turn to the distributional characterisation of $\pi_d$ from \cite{2sat}. For this, let $\cP(0,1)$ denote the set of Borel probability measures on $(0,1)$, and define the operator $\DE_d:\cP(0,1)\to\cP(0,1)$, $\pi\mapsto\hat\pi$ as follows.
Let $\vd^+,\vd^- \sim \Po(d/2)$ and $(\MU_{\pi,i})_{i \geq 1}$ be a sequence of i.i.d. samples from $\pi$. Again, all random variables are assumed to be independent. The image $\hat \pi$ of $\pi$ under $\DE_d$ is then
defined as the distribution of the random variable 
\begin{align}\label{eqdensityEv}
\frac{\prod_{i=1}^{\vd^-}\MU_{\pi,i}}
{\prod_{i=1}^{\vd^-}\MU_{\pi,i}+\prod_{i=1}^{\vd^+}\MU_{\pi,i+\vd^-}}.
\end{align}
If $\delta_{1/2}$ denotes the Dirac measure in $1/2$, then Theorem 1.1 in \cite{2sat} states that for any $d \in (0,2)$, $\lim_{\ell \to \infty} \DE_d^{(\ell)}(\delta_{1/2}) = \pi_d$, where $\DE_d^{(\ell)}$ denotes the $\ell$-fold application of $\DE_d$. In particular, we have the distributional identity
\begin{align}\label{eq:distri_pi}
    \DE_d(\pi_d) = \pi_d.
\end{align}
The convergence proof in \cite{2sat} goes through the analysis of a different, related operator, which we define next, as this point of view will also be more convenient for our analysis. 
For this, let $\cP(\RR)$ denote the set of Borel probability measures on $\RR$, equipped with the weak topology. Denoting by $\cW_2(\RR)$ the set of probability measures $\rho\in\cP(\RR)$ with a finite second moment, we may adorn $\cW_2(\RR)$ with the Wasserstein metric
\begin{align}
W_2(\rho, \rho') := \inf\cbc{\bc{\int_{\RR^2}|x-y|^2\dd \gamma(x,y)}^{1/2}: \gamma \text{ has marginals } \rho, \rho'}, \qquad \rho, \rho' \in 
\zhu{\cW_2(\RR)}.
%\cP(\RR).
\end{align}
It is well-known that $(\cW_2(\RR), W_2)$ is a (non-empty) complete metric space~\cite{AGS,V09}.

The second operator $\LDE_d:\cP(\RR)\to\cP(\RR)$, $\rho\mapsto\hat\rho$ is defined as follows. Let $\vd \sim \Po(d)$, $(\vec \vartheta_{\rho,i})_{i \geq 1}$ be a sequence of i.i.d. samples from $\rho$, and $(\vs_i)_{i \geq 1}, (\vs_i')_{i \geq 1}$ be sequences of i.i.d. Rademacher random variables with parameter $1/2$. All these sequences and $\vd$ are assumed to be independent of each other. The image $\hat \rho$ of $\rho$ under $\LDE_d$ is then
defined as the distribution of the random variable
\begin{align}\label{Eq_BTreeOperator}
\sum_{i=1}^{\vd}
		\vs_i\log\frac{1+\vs_i'\tanh(\vec \vartheta_{\rho,i}/2)}2. 
\end{align}

\begin{lemma}[{\cite[Lemma 4.3]{2sat}}]\label{Lemma_BP1}
For $d<2$, $\LDE_d$ is a contraction on the space $(\cW_2(\RR), W_2)$ with unique fixed point $\rho_d$.
\end{lemma}

In terms of the fixed point $\rho_d$ from Lemma~\ref{Lemma_BP1}, \cite[Lemma 4.2]{2sat} in combination with the continuous mapping theorem imply that  $\pi_d$ can be obtained as
\begin{align}\label{eq:pi_from_rho}
    \pi_d = \psi(\rho_d), 
\end{align}
where $\psi(\rho_d)$ denotes the pushforward measure of $\rho_d$ under the function $\psi$ from \eqref{eqll}\footnote{That is, $\psi(\rho_d)(A) = \rho_d(\psi^{-1}(A))$ for all Borel sets $A \subset (0,1)$.}. 
In Section~\ref{sec-as-conv-root-marignal}, we will use Lemma~\ref{Lemma_BP1} in combination with the relation $\pi_d = \psi(\rho_d)$ to argue that indeed, the random variable $\MU_o$ from Proposition~\ref{Prop_asconvergence} has distribution $\pi_d$.

\section{Almost sure convergence of the root marginal in a truncated Galton-Watson tree}\label{sec-as-conv-root-marignal}

The basis for both Propositions \ref{lem_atoms1} and \ref{lem_atoms2} is the construction of a random variable $\MU_o$ with distribution $\pi_d$ that satisfies the distributional identity \eqref{eq:distri_pi} in an almost sure sense. For this, we will make use of the fact that the local limit of random $2$-SAT is given by a multi-type branching process, and construct $\MU_o$ exploiting the recursive nature of the latter.

\subsection{Multi-type Galton-Watson process}
Consider the following $5$-type Galton-Watson process $\vT$ with one \textit{variable} type and four \textit{clause} types $(1,1), (1,-1), (-1,1)$ and $(-1,-1)$, as described in \cite[Section 2.2]{2sat}: The root node $o$ of the Galton-Watson tree $\vT$ is always a variable node. Variable nodes give rise to a $\Po(d/4)$-number of clause descendants of type $(s_1,s_2)$ for each $(s_1,s_2) \in \{-1,1\}^2$, independently for different clause types and distinct variables. Clause nodes of any of the four types have exactly one offspring of the variable type. As the survival probability of $\vT$ coincides with that of single-type $\Po(d)$-Galton Watson tree, $\vT$ is finite with probability $1$ for $d \leq 1$, while it is infinite with positive probability for $d >1$.  

For any $\ell \geq 0$, let $\vT^{(\ell)}$ denote truncation of $\vT$ at level $\ell$. In particular, $\vT^{(\ell)}$ is a finite tree.
 The Galton-Watson tree $\vT^{(2\ell)}$ truncated at an even level can be naturally regarded as a $2$-SAT formula, where variable nodes correspond to Boolean variables and clause nodes correspond to clauses. Here, type $(s_1, s_1') \in \{-1,1\}^2$ indicates that the parent variable of the clause appears with sign $s_1$ in the present clause, while its child appears with sign $s_1'$. We will also use the notation $\sign(x,a) \in \{-1,1\}$ to indicate the sign with which variable $x$ appears in clause $a$. Thus, if $a=(s_a,s_a')$ is a clause node with parent variable $u$ and child variable $v$, $\sign(u,a)=s_a$ and $\sign(v,a) = s_a'$. Indeed, \cite[Proposition 2.6]{2satclt} shows that as $n \to \infty$, $\vT$ accurately captures the local neighbourhood structure of a uniformly chosen variable in $\PHI_n$.

To formulate our almost sure decomposition, consider a sequence $(\vT_j)_{j \geq 1}$ of i.i.d. copies of the $5$-type Galton Watson tree as described above with respective roots $(o_j)_{j \geq 1}$, and let $(\vec a_{o_j})_{j \geq 1} = ((\vec s_{\zhu{j}}, \vec s_{\zhu{j}}'))_{j \geq 1}$ be an i.i.d. sequence of clauses, independent of $(\vT_j)_{j \geq 1}$. We then specifically regard $\vT$ as being built by first generating $\vd_o \sim \Po(d)$ and subsequently joining the roots $o_1, \ldots, o_{\vd}$ of the first $\vd$ trees in the sequence $(\vT_j)_{j \geq 1}$ to a new root variable $o$ via the clauses $\vec a_{o_1}, \ldots, \vec a_{o_{\vd}}$. For an illustration of this procedure, see Figure~\ref{fig:vT} below. As above, we write $\vT^{(2\ell)}_j$ for the truncation of the subtree $\vT_j$ at level $2\ell$. 

\begin{figure}[b]
\begin{tikzpicture}[-,>=stealth',level/.style={sibling distance = 3cm/#1,
  level distance = 1.5cm}] 
  \filldraw[fill=red!20] (-4.5,-6.5) rectangle (-1.5,-2.5) (-2,-3) node{$\vT_1$};
  \filldraw[fill=red!20] (-0.75,-3.5) rectangle (0.75,-2.5) (0.5,-3) node{$\vT_2$};
  \filldraw[fill=red!20] (1,-6.5) rectangle (5,-2.5) (4.5,-3)node{$\vT_3$};
\node [arn_n] {$o$}
    child{ node [arn_r] {$\vec a_{o_1}$} 
            child{ node [arn_n] {$o_1$} 
            	child{ node [arn_r] {} child{ node [arn_n] {}
             edge from parent node[left]{{\scriptsize $+1$}}}
             edge from parent node[above left]{{\scriptsize $+1$}}} 
						child{ node [arn_r] {}
							child{ node [arn_n] {}
       edge from parent node[left]{{\scriptsize $-1$}}}
                            edge from parent node[left]{{\scriptsize $-1$}}}
                     child{ node [arn_r] {}
							child{ node [arn_n] {}
       edge from parent node[left]{{\scriptsize $-1$}}}
                            edge from parent node[above right]{{\scriptsize $+1$}}}
            edge from parent node[left]{{\scriptsize $-1$}}}       
            edge from parent node[above left]{{\scriptsize $+1$}}
    }
    child{ node [arn_r] {$\vec a_{o_2}$} 
            child{ node [arn_n] {$o_2$} 
    edge from parent node[left]{{\scriptsize $+1$}}}
    edge from parent node[left]{{\scriptsize $-1$}}}
    child{ node [arn_r] {$\vec a_{o_3}$} 
            child{ node [arn_n] {$o_3$} 
            	child{ node [arn_r] {} child{ node [arn_n] {}
             edge from parent node[left]{{\scriptsize $+1$}}}
             edge from parent node[above left]{{\scriptsize $+1$}}} 
             child{ node [arn_r] {} child{ node [arn_n] {}
             edge from parent node[left]{{\scriptsize $+1$}}}
             edge from parent node[left]{{\scriptsize $-1$}}}
							child{ node [arn_r] {}
							child{ node [arn_n] {}
       edge from parent node[left]{{\scriptsize $-1$}}}
                            edge from parent node[right]{{\scriptsize $-1$}}}
                     child{ node [arn_r] {}
							child{ node [arn_n] {}
       edge from parent node[left]{{\scriptsize $+1$}}}
                            edge from parent node[above right]{{\scriptsize $+1$}}}
            edge from parent node[left]{{\scriptsize $-1$}}}       
            edge from parent node[above left]{{\scriptsize $-1$}}
    }

; 

\end{tikzpicture}
  \caption{Illustration of the tree $\vT$ obtained from joining the roots $o_1, o_2,o_3$ of the $\vd_o=3$ trees $\vT_1,\vT_2,\vT_3$ to a new root variable $o$ via the clauses $\vec a_{o_1}, \vec a_{o_2}, \vec a_{o_3}$ (truncated at level $4$).}
  \label{fig:vT}
\end{figure}
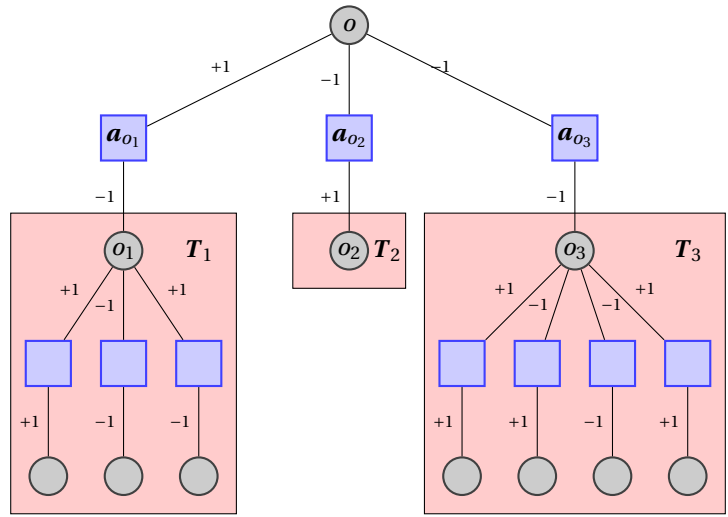

Using the natural correspondence between $\vT^{(2\ell)}$ and $2$-SAT formulas, let $S(\vT^{(2\ell)})$ be the set of all satisfying assignments of the tree formula $\vT^{(2\ell)}$. As any such `tree formula' has at least one satisfying assignment, for any $\ell \geq 0$, the uniform distribution $\mu_{\vT^{(2\ell)}}$ over solutions to $\vT^{(2\ell)}$ from \eqref{eqGibbs} is well-defined.  
Therefore, as before, we use the notation $\mu_{T}(\SIGMA \in A)$ to denote the probability that a sample $\SIGMA$ from $\mu_{T}$ takes its value in some measurable set $A$.  The main result of this section is the following:

\begin{proposition}\label{Prop_asconvergence}
Let $d \in (0,2)$. Then all of the following hold:
\begin{itemize}
    \item[i)] There exist integrable random variables $\vec \mu_o, (\vec \mu_{o_j})_{j \geq 1}$ such that
\begin{align}\label{eqProp_asconvergence1}
\mu_{\vT^{(2\ell)}}(\SIGMA_o =1) \longrightarrow \vec \mu_o \qquad \text{and} \qquad \mu_{\vT_j^{(2\ell)}}(\SIGMA_{o_j}=1) \longrightarrow \vec \mu_{o_j} \quad \text{for all } j \geq 1 \qquad \qquad \text{ as } \ell \to \infty,
\end{align}
almost surely and in $\cL^1$.
\item[ii)] $(\vec \mu_{o_j})_{j \geq 1}$ is a sequence of independent random variables that is also independent of $\vd_o$ and $(\va_{o_j})_{j \geq 1}$.
\item[iii)] Let $\vec D_{\pm} = \{j \in [\vd_o]:  \vec s_j = \pm 1\}$ denote the set of subtrees of the root $o$ with $\sign(o, \vec a_{o_j}) = \pm 1$.  
Then almost surely,
\begin{align}\label{eqProp_asconvergence2}
\vec \mu_o =  \frac{\prod_{i\in \vec D_{-}}\MU_{o_i}^{(1+\vec s_i')/2}(1-\MU_{o_i})^{(1-\vec s_i')/2}}
{\prod_{i\in \vec D_{-}}\MU_{o_i}^{(1+\vec s_i')/2}(1-\MU_{o_i})^{(1-\vec s_i')/2}+\prod_{i\in \vec D_{+} }\MU_{o_i}^{(1+\vec s_i')/2}(1-\MU_{o_i})^{(1-\vec s_i')/2}}.
\end{align}
Moreover, all random variables $\vec \mu_o, (\vec \mu_{o_j})_{j \geq 1}$ from i) have distribution $\pi_d$.
\end{itemize}
\end{proposition}

\subsection{Proof of Proposition \ref{Prop_asconvergence}}
The remainder of this section is devoted to the proof of Proposition \ref{Prop_asconvergence}. Instead of working with the marginals $\mu_{\vT^{(2\ell)}}(\SIGMA_o =1)$, it will be easier to work with their log-likelihood ratios. To this end, for $\ell \geq 0$ and a variable node $y$ of $\vT^{(2\ell)}$, let $\vT_y^{(2\ell)}$ denote \zhu{the sub-formula of $\vT$ rooted at the variable node $y$, truncated at depth $2\ell$}, and define
\begin{align}\label{eq:def_theta}
    \vec \vartheta_y^{(\ell)} = \log(\mu_{\vT_y^{(2\ell)}}(\SIGMA_y=1)/\mu_{\vT_y^{(2\ell)}}(\SIGMA_y=-1)).
\end{align}
This quantity is well-defined as any tree-formula $\vT_y^{(2\ell)}$ has both a satisfying assignment $\SIGMA \in S(\vT_y^{(2\ell)})$ with $\SIGMA_y=1$ and a satisfying assignment $\SIGMA' \in S(\vT_y^{(2\ell)})$ with $\SIGMA'_y=-1$. Moreover, if $\psi$ is defined as in \eqref{eqll}, we have the relation
\begin{align}\label{eq-mu-psi-vtheta}
    \mu_{\vT_y^{(2\ell)}}(\SIGMA_y=1) = \psi(\vec \vartheta_y^{(\ell)}).
\end{align}
Finally, for a variable node $y$ of \zhu{$\vT$}, %$\vT^{(2\ell)}$, 
let $\vd_y$ denote the size of its offspring in \zhu{$\vT$}, %$\vT_y^{(2\ell)}$, 
$y_1, \ldots, y_{\vd_y}$ the associated pendant variables and $\vec a_{y_1}, \ldots, \vec a_{y_{\vd_{y}}}$ its pendant clauses. 
By the recursive nature of marginals on finite tree formulas (see \eqref{eqBPrec}) we have the almost sure relation
\begin{align}\label{eq:BPinGW}
 \vec\vartheta_y^{(\ell)}&= -\sum_{i=1}^{\vd_y}\sign(y,\vec a_{y_i})\log\frac{1+\sign(y_i,\vec a_{y_i})\tanh(\vec\vartheta^{(\ell-1)}_{y_i}/2)}{2}.
\end{align}
The basic idea of the proof of Proposition \ref{Prop_asconvergence} is to re-use a contraction argument from \cite{2sat} on the sequence 
$(\vec \vartheta_o^{(\ell)})_{\ell \geq 1}$.
This gives $\cL^1$-convergence (and thus convergence in probability) of 
$(\vec \vartheta_o^{(\ell)})_{\ell \geq 1}$, together with a speed of convergence which is sufficient to turn the convergence in probability into a.s. convergence, using the first Borel-Cantelli lemma.

\subsubsection{Convergence of the log-likelihood ratios}\label{Sec:Conv_LL}

We  first show that 
$(\vec \vartheta_o^{(\ell)})_{\ell \geq 0}$ and $(\vec \vartheta_{o_j}^{(\ell)})_{\ell \geq 0}$ are $\cL^1$-Cauchy sequences for any $j \geq 1$. Indeed, as the trees $\vT$ and $\vT_j$ are identically distributed Galton-Watson trees, we immediately obtain the following.
\begin{fact} \label{fact:subtrees}
    For any $j \geq 1$, the sequences $(\vec \vartheta_{o}^{(\ell)})_{\ell \geq 0}$ and $(\vec \vartheta_{o_j}^{(\ell)})_{\ell \geq 0}$ are identically distributed.
\end{fact}
Thus, it is sufficient to consider the sequence $(\vec \vartheta_o^{(\ell)})_{\ell \geq 0}$. We first establish the following lemma:

\begin{lemma}\label{lem:successive_theta}
   For any $\ell \geq 1$ and $d \in (0,2)$, we have
   \[\Erw\brk{\abs{\vec \vartheta_o^{(\ell+1)} - \vec \vartheta_o^{(\ell)}}} \leq \frac{d}{2} \cdot \Erw\brk{\abs{\vec \vartheta_o^{(\ell)} - \vec \vartheta_o^{(\ell-1)}}}. \]
\end{lemma}

\begin{proof}
First, using induction, we will check that for all $\ell \geq 0$, $\Erw[|\vec \vartheta_o^{(\ell)}|] < \infty$. For this, observe that $\vec \vartheta_o^{(0)} = 0$, from which the claim follows in the case $\ell=0$. For $\ell \geq 1$, assume that $\Erw[|\vec \vartheta_o^{(\ell-1)}|] < \infty$. Thanks to \eqref{eq:BPinGW},
\begin{align*}
 \Erw\brk{\abs{\vec\vartheta_o^{(\ell)}}}&= \Erw\brk{\abs{\sum_{i=1}^{\vd_o}\sign(o,\vec a_{o_i})\log\frac{1+\sign(o_i,\vec a_{o_i})\tanh(\vec\vartheta^{(\ell-1)}_{o_i}/2)}{2}}} \leq d \Erw\brk{\abs{\log\frac{1+\sign(o_1,\vec a_{o_1})\tanh(\vec\vartheta^{(\ell-1)}_{o_1}/2)}{2}}}.
\end{align*}
Now, since both the derivatives of $x \mapsto \log \frac{1-\tanh(x/2)}{2}$ and of $x \mapsto \log \frac{1+\tanh(x/2)}{2}$ are bounded by one in absolute value, we obtain
\begin{align*}
 \Erw\brk{\abs{\vec\vartheta_o^{(\ell)}}}&\leq  d \bc{\Erw\brk{\abs{\vec\vartheta^{(\ell-1)}_{o_1}}} + |\log 2|} = d \bc{\Erw\brk{\abs{\vec\vartheta^{(\ell-1)}_{o}}} + |\log 2|} < \infty 
\end{align*}
due to Fact \ref{fact:subtrees} and the induction hypothesis.

    Next, fix $\ell \geq 1$ and recall that $\vs_i = \sign(o,\va_{o_i}), \vs_i' = \sign(o_i,\va_{o_i})$ for $i \geq 1$. 
    Thanks to \eqref{eq:BPinGW}, we have
    \begin{align}
        \Erw\brk{\abs{\vec \vartheta_o^{(\ell+1)} - \vec \vartheta_o^{(\ell)}}} &= \Erw\brk{\abs{-\sum_{i=1}^{\vd_o}\vs_i \log\frac{1+\vs_i'\tanh(\vec \vartheta_{o_i}^{(\ell)}/2)}{2} +\sum_{i=1}^{\vd_o}\vs_i \log\frac{1+\vs_i' \tanh(\vec \vartheta^{(\ell-1)}_{o_i}/2)}{2}}}.
    \end{align}
    Due to independence, the last expression can be further simplified as 
    \begin{align}
        \Erw\brk{\abs{\vec \vartheta_o^{(\ell+1)} - \vec \vartheta_o^{(\ell)}}} &= \Erw\brk{\abs{\sum_{i=1}^{\vd_o}\vs_i\log\frac{1+\vs_i'\tanh(\vec \vartheta_{o_i}^{(\ell-1)}/2)}{{1+\vs_i'\tanh(\vec \vartheta^{(\ell)}_{o_i}/2)}} }} \leq d \Erw\brk{\abs{\log\frac{1+\vs_1'\tanh(\vec \vartheta_{o_1}^{(\ell-1)}/2)}{{1+\vs_1'\tanh(\vec \vartheta^{(\ell)}_{o_1}/2)}} }}.
    \end{align}
    Now by an analogous line of computation as in \cite[Proof of Lemma 4.3]{2sat}, as the functions $z\mapsto\log(1+\tanh(z/2))$ and $z\mapsto\log(1-\tanh(z/2))$  are increasing and decreasing, respectively, we obtain
    \begin{align*}
    \abs{\log\frac{1+\tanh(\vec \vartheta_{o_1}^{(\ell-1)}/2)}{{1+\tanh(\vec \vartheta^{(\ell)}_{o_1}/2)}} }    &=
    \abs{\int_{\vec \vartheta_{o_1}^{(\ell)}}^{\vec \vartheta_{o_1}^{(\ell-1)}}\frac{\partial\log(1+\tanh(z/2))}{\partial z}\dd z}
    =\int_{\vec \vartheta_{o_1}^{(\ell)}\wedge \vec \vartheta_{o_1}^{(\ell-1)}}^{\vec \vartheta_{o_1}^{(\ell)}\vee \vec \vartheta_{o_1}^{(\ell-1)}}\frac{1-\tanh(z/2)}{2}\dd z,\\
    \abs{\log\frac{1-\tanh(\vec \vartheta_{o_1}^{(\ell-1)}/2)}{{1-\tanh(\vec \vartheta^{(\ell)}_{o_1}/2)}} }   &=
    \abs{\int_{\vec \vartheta_{o_1}^{(\ell)}}^{\vec \vartheta_{o_1}^{(\ell-1)}}\frac{\partial\log(1-\tanh(z/2))}{\partial z}\dd z}
    =\int_{\vec \vartheta_{o_1}^{(\ell)}\wedge \vec \vartheta_{o_1}^{(\ell-1)}}^{\vec \vartheta_{o_1}^{(\ell)}\vee \vec \vartheta_{o_1}^{(\ell-1)}}    \frac{1+\tanh(z/2)}{2}\dd z.
    \end{align*}
    Conditioning on $\vec s_1'$, we thus obtain
    \begin{align}\label{eq:lem_successive}
        \Erw\brk{\abs{\vec \vartheta_o^{(\ell+1)} - \vec \vartheta_o^{(\ell)}}} \leq \frac{d}{2} \Erw\brk{\abs{\vec \vartheta_{o_1}^{(\ell)} - \vec \vartheta_{o_1}^{(\ell-1)}}} = \frac{d}{2} \Erw\brk{\abs{\vec \vartheta_{o}^{(\ell)} - \vec \vartheta_{o}^{(\ell-1)}}},
    \end{align}
    where the last equality is once more a consequence of Fact \ref{fact:subtrees}.
\end{proof}

\begin{corollary}\label{cor:convergence_theta1}
 For any $d \in (0,2)$, $(\vec \vartheta_o^{(\ell)})_{\ell \geq 0}$ is an $\cL^1$-Cauchy sequence.
\end{corollary}

\begin{proof}
    By Lemma \ref{lem:successive_theta}, for any $\ell \in \mathbb{N}$, we have
    \begin{align*}
        \Erw\brk{\abs{\vec \vartheta_o^{(\ell+1)} - \vec \vartheta_o^{(\ell)}}} \leq \bc{\frac{d}{2}}^{\ell} \Erw\brk{\abs{\vec \vartheta_{o}^{(1)} - \vec \vartheta_{o}^{(0)}}} =  \bc{\frac{d}{2}}^{\ell} \Erw\brk{\abs{\vec \vartheta_{o}^{(1)}}},
    \end{align*}
as $\vec \vartheta_o^{(0)}$=0. Thus, for arbitrary $\ell_1, \ell_2 \in \NN_0$ with $\ell_1 < \ell_2$, we have
\begin{align} \label{eq:distance}
    \Erw\brk{\abs{\vec \vartheta_o^{(\ell_2)} - \vec \vartheta_o^{(\ell_1)}}} &\leq  \sum_{\ell=\ell_1}^{\ell_2-1} \Erw\brk{\abs{\vec \vartheta^{(\ell+1)}_{o} - \vec \vartheta^{(\ell)}_{o}}} 
    \leq \Erw\brk{\abs{\vec \vartheta^{(1)}_{o}}}\sum_{\ell=\ell_1}^{\ell_2-1} \bc{\frac{d}{2}}^{\ell} \leq \bc{\frac{d}{2}}^{\ell_1} \frac{1 -\bc{d/2}^{\ell_2-\ell_1}}{1 - d/2} \Erw\brk{\abs{\vec \vartheta^{(1)}_{o}}}.
    \end{align}
As $d<2$, this shows that $(\vec \vartheta_o^{(\ell)})_{\ell \geq 1}$ is an $\cL^1$-Cauchy sequence. 
\end{proof}

As $\cL^1$ is complete, Corollary \ref{cor:convergence_theta1} implies that $(\vec \vartheta_o^{(\ell)})_{\ell \geq 0}$ converges in $\cL^1$. We denote the $\cL^1$-limit of $(\vec \vartheta_o^{(\ell)})_{\ell \geq 0}$ by $\vec \vartheta_o$. As a next step, we show that $(\vec \vartheta_o^{(\ell)})_{\ell \geq 0}$ converges also almost surely to $\vec \vartheta_o$.

\begin{corollary} \label{cor:convergence_theta2}
For any  $d \in (0,2)$, $(\vec \vartheta_o^{(\ell)})_{\ell \geq 0}$ converges almost surely to its $\cL^1$-limit $\vec \vartheta_o$.
\end{corollary}

\begin{proof}
  By considering the $\limsup$ as $\ell_2 \to \infty$ in \eqref{eq:distance}, we obtain the bound $\Erw[|\vec \vartheta_o^{(\ell)} -\vec \vartheta_o|] \leq (d/2)^{\ell} \Erw[|\vec \vartheta^{(1)}_{o}|]/(1-d/2)$ for $\ell \geq 0$. For $\eps>0$ arbitrary, Markov's inequality thus implies that
\begin{align*}
    \sum_{\ell=0}^{\infty} \pr\bc{\abs{\vec \vartheta_o^{(\ell)} -\vec \vartheta_o} \geq \eps} \leq  \sum_{\ell=0}^{\infty} \frac{\Erw\brk{\abs{\vec \vartheta_o^{(\ell)} -\vec \vartheta_o}}}{\eps} \leq \frac{\Erw\brk{\abs{\vec \vartheta_o^{(1)}}}}{\eps (1-d/2)}  \sum_{\ell=0}^{\infty} \bc{\frac{d}{2}}^{\ell} < \infty.
    \end{align*}
By the first Borel-Cantelli lemma $|\vec \vartheta_o^{(\ell)}-\vec\vartheta_o|\geq \eps$ only finitely often. As a consequence, $(\vec \vartheta_o^{(\ell)})_{\ell \geq 0}$  converges almost surely to $\vec \vartheta_o$. 
\end{proof}

\begin{corollary} \label{Lem:LL_dist}
  For any $d \in (0,2)$, $\vec \vartheta_o$ satisfies the almost sure relation
   \begin{align}\label{eqProp_asconvergence3}
\vec \vartheta_o =  \sum_{j=1}^{\vec d_o}(- \vec s_j) \cdot \log \frac{1 + \vec s_j' \tanh(\vec \vartheta_{o_j}/2) }{2}.
\end{align} 
Moreover, $\vec \vartheta_o$ has distribution $\rho_d$, where $\rho_d$ is the unique fixed point of $\LDE_d$ in $(\cW_2(\RR), W_2)$ from Lemma~\ref{Lemma_BP1}. 
\end{corollary}

\begin{proof}
    Taking the limit $\ell \to \infty$ in \eqref{eq:BPinGW} for $y=o$ and using the almost sure convergence from Corollary~\ref{cor:convergence_theta2} yields the almost sure relation \eqref{eqProp_asconvergence3}. To identify the distribution of $\vec \vartheta_o$, let $\rho^{(\ell)}$ denote the distribution of $\vec \vartheta_o^{(\ell)}$ for $\ell \geq 0$, and observe that $\vec \vartheta_o^{(0)} = 0 \in \cL^2$. The distribution $\rho^{(\ell)}$ for $\ell \geq 1$ is specified in terms of $\rho^{(\ell-1)}$ as the distribution of the random variable in  \eqref{eq:BPinGW} (again with $y=o$). However, 
    \[ -\sum_{i=1}^{\vd_o}\sign(o,\vec a_{o_i})\log\frac{1+\sign(o_i,\vec a_{o_i})\tanh(\vec\vartheta^{(\ell-1)}_{o_i}/2)}{2} \quad \disteq \quad \sum_{i=1}^{\vd_o}\sign(o,\vec a_{o_i})\log\frac{1+\sign(o_i,\vec a_{o_i})\tanh(\vec\vartheta^{(\ell-1)}_{o_i}/2)}{2},\]
where the distribution of the random variable on the right hand side is $\LDE_d(\rho^{(\ell-1)})$. We thus obtain that $\rho^{(0)} = \delta_0$, $\rho^{(\ell)} = \LDE_d(\rho^{(\ell-1)})$. By   Lemma~\ref{Lemma_BP1}, this sequence of distributions converges to $\rho_d$, from which we obtain that $\vec\vartheta_o \disteq \rho_d$.
\end{proof}
We will revisit the relation \eqref{eqProp_asconvergence3} in Section~\ref{sec:lem_atoms2} during the proof of Proposition~\ref{lem_atoms2}. 

\subsubsection{Proof of Proposition \ref{Prop_asconvergence}}
Ad i): As $( \mu_{\vT^{(2\ell)}}(\SIGMA_o=1) )_{\ell\geq 0} = (\psi(\vec \vartheta_o^{(\ell)}))_{\ell \geq 0}$ with $\psi$ from \eqref{eqll}, almost sure convergence of $( \mu_{\vT^{(2\ell)}}(\SIGMA_o=1) )_{\ell\geq 0}$ immediately follows from Corollary \ref{cor:convergence_theta2}. We write $\MU_o=\psi(\vec \vartheta_o)$ for the almost sure limit. 
    $\cL^1$-convergence of $( \mu_{\vT^{(2\ell)}}(\SIGMA_o=1) )_{\ell\geq 0}$ on the other hand follows from Corollary \ref{cor:convergence_theta1} and the fact that for all $z_1, z_2 \in \mathbb R$ there exists $\xi = \xi(z_1, z_2) \in \RR$ such that
    \[|\psi(z_1) - \psi(z_2)| = |\psi'(\xi)| |z_1 - z_2| \leq \frac{1}{2} |z_1 - z_2|.\]
The same line of reasoning applies to $( \mu_{\vT_j^{(2\ell)}}(\SIGMA_{o_j}=1) )_{\ell\geq 0}$ for $j \geq 1$, in combination with the fact that we only have countably many such random variables.

Ad ii): The stated independence follows immediately from the construction, using independence of the subtrees $(\vT_j)_{j \geq 1}$ as well as the clauses $(\vec a_{o_j})$ and the Poisson variable $\vd$.

Ad iii): 
Again, by the recursive nature of $2$-SAT marginals on finite trees~\eqref{eqBPrec} for any $\ell \geq 2$, we have the almost sure relation
\begin{align}\label{eq:BPinGW1}
\mu_{\vT^{(2\ell)}}(\SIGMA_o =1)  &= \frac{\prod_{i \in \vec D_{-}}\mu_{\vT_i^{(2\ell-2)}}(\SIGMA_{o_i} = \sign(o_i,\vec a_{o_i}))  }{\prod_{i \in \vec D_{-}}\mu_{\vT_i^{(2\ell-2)}}(\SIGMA_{o_i} = \sign(o_i, \vec a_{o_i}))  + \prod_{i \in \vec D_{+}}\mu_{\vT_i^{(2\ell-2)}}(\SIGMA_{o_i} = \sign(o_i, \vec a_{o_i}))}.
\end{align}
Expressing the right hand side of \eqref{eq:BPinGW1} in terms of $\mu_{\vT_i^{(2\ell-2)}}(\SIGMA_{o_i} = 1)$,  the right hand side of \eqref{eq:BPinGW1} reduces to
\begin{align}\label{eq:BPinGW2}
\frac{\prod_{i \in \vec D_{-}}\mu_{\vT_i^{(2\ell-2)}}(\SIGMA_{o_i} = 1)^{\frac{1+ \vec s_i'}{2}} (1-\mu_{\vT_i^{(2\ell-2)}}(\SIGMA_{o_i} = 1))^{\frac{1-\vec s_i'}{2}}}{\prod_{i \in \vec D_{-}} \mu_{\vT_i^{(2\ell-2)}}(\SIGMA_{o_i} = 1)^{\frac{1+ \vec s_i'}{2}} (1-\mu_{\vT_i^{(2\ell-2)}}(\SIGMA_{o_i} = 1))^{\frac{1-\vec s_i'}{2}} + \prod_{i \in \vec D_{+}}\mu_{\vT_i^{(2\ell-2)}}(\SIGMA_{o_i} = 1)^{\frac{1+ \vec s_i'}{2}} (1-\mu_{\vT_i^{(2\ell-2)}}(\SIGMA_{o_i} = 1))^{\frac{1-\vec s_i'}{2}}}.
\end{align}
Taking the limit $\ell \to \infty$ in \eqref{eq:BPinGW2}, together with the almost sure convergence of all involved random variables, yields \eqref{eqProp_asconvergence2}. Moreover, by Fact \ref{fact:subtrees}, all random variables $\MU_o, (\MU_{o_j})_{j \geq 1}$, are identically distributed. Finally, as $\MU_o = \psi(\vec\vartheta_o)$, Corollary~\ref{Lem:LL_dist} together with \eqref{eq:pi_from_rho} imply that $\MU_o$ has distribution $\psi(\rho_d) = \pi_d$.

\section{Proof of Proposition \ref{lem_atoms1}} \label{sec:lem_atoms1}

We prove Proposition \ref{lem_atoms1} by a closer inspection of \textit{finite (variable-)rooted} $2$-SAT tree factor graphs $T$ as considered in Section~\ref{subsec:BP_on_trees}. The proof strategy is the following: First, observe that any such tree $T$ has positive probability under the law of the Galton-Watson tree $\vT$, i.e. $\pr(\vT=T) >0$. 
Thanks to Proposition \ref{Prop_asconvergence}, on the event $\{\vT = T\}$, the random variable $\MU_o \sim \pi_d$ is nothing but the root marginal $\mu_T(\SIGMA_o=1)$ of $T$. Secondly, this marginal clearly is a rational number. To argue the converse, namely that every rational number in $(0,1)$ occurs as the root marginal of a finite rooted factor tree, in Subsection \ref{subsec:finite_trees} we make use of the recursive structure of marginals on such trees (see Subsection \ref{subsec:BP_on_trees}). In Subsection \ref{subsec:proof_atoms1} we combine these two observations to conclude the proof of Proposition \ref{lem_atoms1}.

\subsection{Finite rooted trees} \label{subsec:finite_trees}

Let
\begin{align*}
    \mathbb A_{\mathrm{T}} := \cbc{\mu_{T}(\SIGMA_o = 1): T \text{ is a finite rooted tree with root } o}
\end{align*}
be the set of root marginals of $2$-SAT formulae encoded by finite rooted tree factor graphs. As a preparation to Proposition~\ref{prop:marginals_of_trees}, we first prove two small lemmata.

\begin{lemma}\label{lem:impl1}
    $q \in \mathbb A_{\mathrm{T}} \Longrightarrow 1-q \in \mathbb A_{\mathrm{T}}$.
\end{lemma}
\begin{proof}
    Assume that $q \in \mathbb A_{\mathrm{T}}$ and let $T$ be a finite factor tree with root $o$ and $\mu_{T}(\SIGMA_o = 1) = q$. Define $T'$ to be the finite factor tree that has the same graph structure as $T$, but where the original edge labels of $T$ have been multiplied by $-1$. Then $S(T') = \{- \sigma: \sigma \in S(T)\}$, and in particular both formulas have the same number of solutions. Correspondingly the proportion of solutions of $T'$ where the root is set to $1$ is identical to the proportion of solutions of $T'$ where the root is set to $-1$. In other words, $\mu_{T'}(\SIGMA_o=1) = \mu_{T}(\SIGMA_o=-1) = 1 - \mu_{T}(\SIGMA_o=1) = 1-q$.
\end{proof}

\begin{lemma}\label{lem:impl2}
    $p,q \in \mathbb A_{\mathrm{T}} \Longrightarrow \frac{p}{p+q} \in \mathbb A_{\mathrm{T}}$. 
\end{lemma}
\begin{proof}
Assume that $p,q \in \mathbb A_{\mathrm{T}}$ and let $T_1, T_2$ be two finite factor trees with roots $o_1, o_2$ and $\mu_{T_1}(\SIGMA_{o_1}=1) =p$ and $\mu_{T_2}(\SIGMA_{o_2}=1) =q$. Consider the tree $T$ that is built from $T_1$ and $T_2$ by joining them through a new root $o$ as follows. The root $o$ has exactly one $(-,+)$ child and one $(+,+)$ child. At the $(-,+)$ child, we append the tree $T_1$ at $o_1$. At the $(+,+)$ child, we append the tree $T_2$ at $o_2$. According to~\eqref{eqBPrec},  $\mu_T(\SIGMA_o=1) = \frac{p}{p+q}$.
\end{proof}

We are now in the position to show that every rational number occurs as the root marginal of a rooted 2-SAT factor tree.

\begin{proposition}\label{prop:marginals_of_trees}
    We have
    \[\mathbb A_{\mathrm{T}} = \mathbb{Q} \cap (0,1).\]

\end{proposition}

\begin{proof}
First, observe that $\mathbb A_{\mathrm{T}} \subseteq \mathbb{Q} \cap [0,1]$, as for any finite factor tree $T$ with root $o$, $\mu_{T}(\SIGMA_o = 1)$ is nothing but the fraction of solutions $\sigma$ of the finite tree formula $T$ in which the root is set to $1$. Moreover, it cannot happen that all solutions $\sigma$ have $\sigma_o=1$ (or $\sigma_o = -1$), as assigning the value $-1$ to $o$ can always be completed to a valid solution of $T$ by tracing the implications down the tree. Thus, $\mathbb A_{\mathrm{T}} \subseteq \mathbb{Q} \cap (0,1)$.

The basic proof idea for the opposite inclusion $\mathbb A_{\mathrm{T}} \supseteq \mathbb{Q} \cap (0,1)$ is to build deeper and deeper trees with more complex root marginals, using smaller subtrees along with the recursive computation of marginals that is possible in trees. Using Lemmata~\ref{lem:impl1} and~\ref{lem:impl2}, we can now proceed by means of a nested induction.

\textbf{Induction beginning (\textbf{I.B.}):} We show that $\{\frac{1}{k}: k >1\} \subseteq \mathbb A_{\mathrm{T}}$ by induction on $k \geq 2$. \\
\textit{Induction beginning: $1/2 \in \mathbb{A}_T$}. Let $T$ be the tree of height $0$ that just consists of an isolated root variable $o$. We then have $\mu_T(\SIGMA_o=1)= \frac{1}{2}$, so $\frac{1}{2} \in \mathbb A_{\mathrm{T}}$.\\
\textit{Induction step: $1/k \in \mathbb{A}_T\Rightarrow 1/(k+1) \in \mathbb{A}_T$}. Assume that $1/k \in \mathbb A_{\mathrm{T}}$ and let $T$ be a finite rooted factor tree with root $o$ such that $1/k = \mu_T(\SIGMA_o=1)$. Consider the tree $T'$ that is built from $T$ as follows: At level zero, there is the root variable $o'$. It has exactly one $(-,+)$ child. At the $(-,+)$ child, we append the tree $T$ at $o$.  According to~\eqref{eqBPrec},
\[ \mu_{T'}(\SIGMA_o=1) = \frac{\frac{1}{k}}{\frac{1}{k} + 1} = \frac{1}{k+1}.\]
Thus, $\frac{1}{k+1} \in  \mathbb A_{\mathrm{T}}$. This concludes the proof of the outer induction beginning \textbf{I.B.}.

\textbf{Induction step.} We next show that for any fixed $r \geq 1$, $\{\frac{r}{k}: k > r\} \subseteq \mathbb A_{\mathrm{T}}$ implies $\{\frac{r+1}{k}: k > r+1\} \subseteq \mathbb A_{\mathrm{T}}$.\\
Therefore, the \textbf{induction hypothesis (I.H.)} is that for some $r \geq 1$, $\{\frac{r}{k}: k > r\} \subseteq \mathbb A_{\mathrm{T}}$ (\textbf{I.H.}). We will show that $\{\frac{r+1}{k}: k > r+1\} \subseteq \mathbb A_{\mathrm{T}}$ by induction on $k > r+1$.\\ 
\textit{Induction beginning: $k=r+2$}. As $\{\frac{1}{k}: k >1\} \subseteq \mathbb A_{\mathrm{T}}$ according to the (outer) induction beginning \textbf{I.B.}, $1/(r+2) \in \mathbb{A}_{\mathrm{T}}$. This implies that $(r+1)/(r+2) = 1 - 1/(r+2)\in \mathbb{A}_{\mathrm{T}}$, using Lemma~\ref{lem:impl1}. \\
\textit{Induction step: $(r+1)/k \in \mathbb{A}_T\Rightarrow (r+1)/(k+1) \in \mathbb{A}_T$}. Assume that $(r+1)/k \in \mathbb A_{\mathrm{T}}$ for some $k \geq r+2$. By the (outer) induction hypothesis \textbf{I.H.}, $r/k \in \mathbb{A}_{\mathrm{T}}$. As a consequence of Lemma~\ref{lem:impl1}, also $(k-r)/k \in \mathbb A_{\mathrm{T}}$. Now Lemma~\ref{lem:impl2} implies that  
\[\frac{r+1}{k+1} = \frac{\frac{r+1}{k}}{\frac{r+1}{k} + \frac{k-r}{k}} \in \mathbb A_{\mathrm{T}}.\]
This concludes the induction, which in turn clearly shows that any rational number in $(0,1)$ is contained in $\mathbb{A}_T$.
\end{proof}

\subsection{{Proof of Proposition \ref{lem_atoms1}}} \label{subsec:proof_atoms1}
We now have all our ducks in a row. Let $q \in \mathbb Q \cap (0,1)$. According to  Proposition \ref{prop:marginals_of_trees}, $q = \mu_T(\SIGMA_o=1)$ for some finite 2-SAT factor tree $T$ with root $o$. Moreover $\pr(\vT=T)>0$. On the event $\{\vT = T\}$, $\MU_o = \mu_T(\SIGMA_o=1) = q$. Thus $\pi_d(\{q\}) \geq \pr(\vT=T) > 0$, from which we conclude that $\mathbb Q \cap (0,1) \subseteq \mathbb A_{\mathrm{p.p.}}$.

\section{{Proof of Proposition~\ref{lem_atoms2}}}  \label{sec:lem_atoms2}
To prove the inclusion $\mathbb{A}_{\mathrm{p.p.}} \subseteq \mathbb Q \cap (0,1)$ as well as the existence of a non-trivial continuous part in $\pi_d$ for $d \in (1,2)$, in Section \ref{Sec:Inf_trees_no_atoms} we analyse the distribution of the random variable $\MU_o$ from  Proposition~\ref{Prop_asconvergence} on the event $\{|\vT| = \infty\}$, which has non-zero probability for $d \in (1,2)$. As finite trees cannot give rise to a continuous distribution, this is a natural candidate for the continuous part of $\pi_d$, and in fact, Lemma~\ref{lem:inf_trees_noatoms} below shows that this part is purely continuous. In Section \ref{Sec:Supp_inf}, we will then use a similar line of argument to show that the support of the continuous part $\pi_{d,\mathrm{c}}$ is $[0,1]$. The proof of Proposition~\ref{lem_atoms2} is given in the final Section~\ref{Sec:Proof_atoms2}.

\subsection{Infinite trees have no atoms} \label{Sec:Inf_trees_no_atoms}
The main step in the proof of Proposition~\ref{lem_atoms2} is Lemma~\ref{lem:inf_trees_noatoms}, which we state and prove in this section.

\begin{lemma}\label{lem:inf_trees_noatoms}
    For any $d \in (1,2)$ and $x \in (0,1)$,
    \[\pr\bc{\MU_o = x \vert |\vT|= \infty} = 0.\]
\end{lemma}

To prove Lemma \ref{lem:inf_trees_noatoms}, it will, once more, be more convenient to work with log-likelihood ratios. To this end, recall the random variables $\vec \vartheta_o, (\vec\vartheta_{o_j})_{j \geq 1}$ from Section~\ref{Sec:Conv_LL} as well as the almost sure relation \eqref{eqProp_asconvergence3}. We then consider the maximal point mass 
\[\vartheta^{\ast} = \sup_{z \in \RR} \pr\bc{\vec \vartheta_o = z \vert |\vT| = \infty} \in [0,1]
\]
of $\vec \vartheta_o$ conditionally on having sampled an infinite tree. Observe that the supremum is attained in some value $z^{\ast} \in \RR$, i.e. $\pr\bc{\vec \vartheta_o = z^{\ast} \vert |\vT| = \infty} = \vartheta^{\ast}$. 
We further define  
\begin{align}
    \mathcal T_{o, \mathrm{fin}} = \cbc{i \in [\vd_o]: |\vT_i|<\infty}, \quad  \mathcal T_{o, \mathrm{inf}} = \cbc{i \in [\vd_o]: |\vT_i|=\infty}.
\end{align}
We may then decompose the sum over subtrees \eqref{eqProp_asconvergence3} into a contribution from the infinite and the finite subtrees: 
\begin{align}\label{eq:lldecom_inf}
     \vec\vartheta_o = -\sum_{i\in \mathcal T_{o, \mathrm{fin}}}\vec s_i\log\frac{1+\vec s_i'\tanh(\vec\vartheta_{o_i}/2)}{2} - \sum_{i\in \mathcal T_{o, \mathrm{inf}}}\vec s_i\log\frac{1+\vec s_i'\tanh(\vec\vartheta_{o_i}/2)}{2} =: \efin + \einf. 
\end{align}
This decomposition is well-defined as all single subtree contributions are finite almost surely. 
The first observation underlying the proof of Lemma~\ref{lem:inf_trees_noatoms} is an application of the following fact to the decomposition \eqref{eq:lldecom_inf}.
 \begin{lemma}\label{coro-sum-atom}
    Let $\vX$ and $\vY$ be two independent, real-valued random variables. 
    Then 
    \begin{align}\label{eq-sum-atom}
       \sup_{z \in \RR} \pr(\vX+\vY=z) \leq \sup_{z \in \RR} \pr(\vY=z).
    \end{align}
 \end{lemma}
 \begin{proof}
Let $\pr_{\vX}$ denote the distribution of $\vX$. Then for any $z \in \RR$,
\begin{align}\label{eq:sum-atoms}
    \pr\bc{\vX+\vY=z} =\int \pr\bc{\vY=z-x} \,d\pr_{\vX}(x)
     \leq \sup_{z' \in \RR} \pr(\vY=z') \int 1 \,d\pr_{\vX}(x) = \sup_{z' \in \RR} \pr(\vY=z').
     \end{align}
 \end{proof}

\zhu{Note that in the decomposition $\vec\vartheta_o = \efin + \einf$ from \eqref{eq:lldecom_inf}, 
$\efin$ and $\einf$ are conditionally independent given $\{|\vT|= \infty\}$. 
Indeed, among all the grandchildren $o_1,\ldots,o_{\vd_o}$ of the root $o$, by Poisson thinning, 
the numbers $\vd_{o,\mathrm{inf}}$ and $\vd_{o,\mathrm{fin}}$ of grandchildren $i$ satisfying 
$\{|\vT_{o_i}|= \infty\}$ and $\{|\vT_{o_i}|< \infty\}$, respectively, 
follow independent Poisson distributions.
Since the condition $\{|\vT|= \infty\}$ only restricts $\vd_{o,\mathrm{inf}}>0$, 
$\vd_{o,\mathrm{fin}}$ and $\vd_{o,\mathrm{inf}}$ remain conditionally independent, 
which implies that $\efin$ and $\einf$ are also conditionally independent. Then,}
an application of Lemma~\ref{coro-sum-atom} to the decomposition $\vec\vartheta_o = \efin + \einf$ conditionally on the event $\{|\vT|= \infty\}$, yields the bound
 \begin{align}\label{eq:inf_first_bound}
\vartheta^{\ast} =  \sup_{z \in \RR} \pr(\vec \vartheta_o =z \mid |\vT| = \infty) \leq  \sup_{z \in \RR} \pr(\einf =z \mid |\vT| = \infty).
\end{align}
The next lemma establishes a reverse bound from which it will be immediate that $\vartheta^{\ast}=0$.
\begin{lemma}\label{lem-atom-GAMMA'-no-more-ETA}
For $d \in (1,2)$,
\begin{align*}
  \sup_{z \in \RR}  \pr\bc{\einf=z \mid |\vT| = \infty}\leq \vartheta^{\ast}/2.
\end{align*}
\end{lemma}

\begin{proof}[Proof of \Cref{lem-atom-GAMMA'-no-more-ETA}] 
Fix $t \in \NN$ as well as $\emptyset \not= I \subseteq [t]$.  By the law of total probability, it is sufficient to show that for any such choice of $t$ and $I$, for all $z \in \RR,$ 
\begin{align}\label{eq:inf_subtree_count}
     \pr\bc{\einf=z \mid \vec \vd_o = t, \forall i \in I:  |\vT_i|=\infty, \forall j \in [t]\setminus I: |\vT_i|<\infty} \leq \vartheta^{\ast}/2.
\end{align}
To this end, let $i_\ast = \min I$ be the `first' infinite subtree. As a first step, we show that 
\begin{align}\label{eq:single_inf}
    \pr\bc{-\vs_{i_\ast} \log\frac{1+\vs_{i_\ast}' \tanh(\vec\vartheta_{o_{i_\ast}}/2)}{2} = z \Big \vert \vec \vd_o = t, \forall i \in I:  |\vT_i|=\infty, \forall j \in [t]\setminus I: |\vT_i|<\infty } \leq \vartheta^\ast/2.
\end{align}
For now, assume that $z>0$. Given the number $t$ of subtrees of the root $o$ as well as the indices $I$ of the infinite subtrees, the subtrees are independent. Recalling the mutually inverse functions $\psi, \varphi$ defined in \eqref{eqll} and that $z>0$, the left hand side of \eqref{eq:single_inf} becomes
\begin{align}\label{eq:cond_si}
    & \pr\bc{-\vs_{i_\ast} \log\bc{\psi(\vs_{i_\ast}'\vec\vartheta_{o_{i_\ast}})} = z \big \vert \vec \vd_o = t, \forall i \in I:  |\vT_i|=\infty, \forall j \in [t]\setminus I: |\vT_i|<\infty } \nonumber \\
    & \qquad =\pr\bc{\vs_{i_\ast} =1} \pr\bc{\vec\vartheta_{o_{i_\ast}} = \vs_{i_\ast}' \varphi\bc{\eul^{-z}}\big \vert |\vT_{i_\ast}|=\infty} = \frac{1}{2} \pr\bc{\vec\vartheta_{o_{i_\ast}} = \vs_{i_\ast}' \varphi\bc{\eul^{-z}}\big \vert |\vT_{i_\ast}|=\infty},
\end{align}
where we have conditioned on $\{\vs_{i_\ast} =1\}$ for $z>0$, as $ \log(\psi(\vs_{i_\ast}'\vec\vartheta_{o_{i_\ast}})) < 0$ almost surely.
On the other hand, the distribution of $\vec \vartheta_{o_{i_\ast}}$ given $\{|\vT_{i_\ast}|=\infty\}$ agrees with the distribution of $\vec\vartheta_o$ given $\{|\vT|=\infty\}$. Thus, we obtain \eqref{eq:single_inf} for $z>0$. Of course, the case $z<0$ can be treated analogously by conditioning on $\vs_{i_\ast}=-1$ in \eqref{eq:cond_si}. Finally, in the case $z=0$, the first conditional probability in \eqref{eq:cond_si} is $0$, since given $\{|\vT_{i_\ast}|=\infty$\}, $\vec\vartheta_{o_{i_\ast}}<\infty$ almost surely. Otherwise, as $\pr(|\vT|=\infty)>0$ for $d>1$, and using again that the distribution of $\vec \vartheta_{o_{i_\ast}}$ given $\{|\vT_{i_\ast}|=\infty\}$ agrees with the distribution of $\vec\vartheta_o$ given $\{|\vT|=\infty\}$, $\pi_d$ would have an atom in $1$, which is ruled out by Remark~\ref{rem:boundary}. Therefore,  $-\vs_{i_\ast} \log(\psi(\vs_{i_\ast}'\vec\vartheta_{o_{i_\ast}})) \neq 0$, so that \eqref{eq:single_inf} indeed holds for all $z \in \RR$.

The claim now follows from another application of Lemma~\ref{coro-sum-atom}: To this end, given the number $t$ of subtrees of the root $o$ as well as the indices $I$ of the infinite subtrees, we decompose $\einf$ further as
    \[\einf =-\vec s_{i_\ast}\log\frac{1+\vec s_{i_\ast}'\tanh(\vec\vartheta_{o_{i_\ast}}/2)}{2} + \sum_{i \in I\setminus\{i_\ast\}} (-\vec s_i)\log\frac{1+\vec s_i'\tanh(\vec\vartheta_{o_i}/2)}{2} =: \einf^{(1)} + \einf^{(2)}.\]
Using conditional independence of the subtrees and \eqref{eq:single_inf}, we thus obtain from Lemma~\ref{coro-sum-atom} that
\begin{align*}
  \pr\bc{  \einf^{(1)} + \einf^{(2)}\mid \vec \vd_o = t, \forall i \in I:  |\vT_i|=\infty, \forall j \in [t]\setminus I: |\vT_i|<\infty}   \leq \vartheta^{\ast}/2,
\end{align*}
and therefore \eqref{eq:inf_subtree_count}.
\end{proof}

\begin{proof}[Proof of Lemma \ref{lem:inf_trees_noatoms}]
It follows directly from the combination of \eqref{eq:inf_first_bound} and \Cref{lem-atom-GAMMA'-no-more-ETA} that
\begin{align*}
0 \leq \vartheta^{\ast} =  \sup_{z \in \RR} \pr(\vec \vartheta_o =z \vert |\vT| = \infty)  \leq \sup_{z \in \RR} \pr(\einf =z \vert |\vT| = \infty)\leq \vartheta^{\ast}/2,
\end{align*}
which can hold only if $\vartheta^{\ast}=0$. This means that the law of $\vec \vartheta_o$ conditional on $\{|\vT|=\infty\}$ has no atoms. As $\MU_o = \psi(\vec \vartheta_o)$ and $\psi$ is injective, the claim follows.
\end{proof}

\subsection{Infinite tree support}\label{Sec:Supp_inf}
It follows from Proposition~\ref{prop:marginals_of_trees} and Lemma~\ref{lem:inf_trees_noatoms} that the restriction of the distribution of $\MU_o$ to the event $\{|\vT|< \infty\}$ is purely discrete, while its restriction to $\{|\vT|= \infty\}$ is purely continuous. 
\Cref{prop:marginals_of_trees} shows that the \zhu{pure point} support of the discrete measure $A \mapsto \pr(\MU_o \in A, |\vT| < \infty)$ is $\mathbb Q \cap (0,1)$. We next identify the support of the continuous measure $A \mapsto \pr(\MU_o \in A, |\vT| = \infty)$. 

\begin{lemma}\label{lem:inf_tree_supp}
For $d \in (1,2)$, $\pi_d$ has a non-trivial continuous part $\pi_{d,\mathrm{c}}$ with $\supp(\pi_{d,\mathrm{c}}) = [0,1]$.
\end{lemma}
\begin{proof}
As discussed at the beginning of this subsection, the continuous part $\pi_{d,\mathrm{c}}$ of $\pi_d$ is given by the restriction of the distribution of $\MU_o$ to the event $\{|\vT|= \infty\}$. In particular, as $|\vT|=\infty$ with strictly positive probability for $d \in (1,2)$, $\pi_{d,\mathrm{c}} \not\equiv 0$.

To establish that $\supp(\pi_{d,\mathrm{c}}) = [0,1]$, we will once more analyse the distribution of $\vec \vartheta_o$ given $\{|\vT|=\infty\}$ and show that its support is the whole real line $\RR$. To this end, we will again condition on a specific tree structure, which is sufficient for our purposes.

Fix $a, b \in \RR$ with $a<b$; we will then argue that $\Pr\bc{\vec\vartheta_o\in [a,b]\vert \vd_o=2, |\vT_1|=\infty, |\vT_2| < \infty}>0$.
To do so, recall the decomposition $\vec\vartheta_o = \efin + \einf$ of the log-likelihood ratio into a contribution from the finite and the infinite subtrees of the root $o$ from \eqref{eq:lldecom_inf}. As mentioned previously, $\Pr\bc{\efin\in \{\pm \infty\} \vert \vd_o=2, |\vT_1|=\infty, |\vT_2| < \infty}=0$ as a consequence of conditional independence and Proposition~\ref{prop:marginals_of_trees}. Similarly, $\Pr\bc{\einf\in \{\pm \infty\} \vert \vd_o=2, |\vT_1|=\infty, |\vT_2| < \infty}=0$, as otherwise, using that $\efin \in \RR$ almost surely, we could write $\vec \vartheta_o = \efin + \einf$ and it would follow that $\pr(\vec\vartheta_o \in \{\pm \infty\})>0$. This however contradicts Remark~\ref{rem:boundary}. 
As a consequence, there exists $c \in \RR$ such that
\begin{align*}
    \Pr\bc{\einf\in \brk{c - \frac{b-a}{4},c+\frac{b-a}{4}}\big\vert \vd_o=2, |\vT_1|=\infty, |\vT_2| < \infty}>0.
\end{align*}
On the other hand, given $\{ \vd_o=2, |\vT_1|=\infty, |\vT_2| < \infty\}$, $\vec\vartheta_{o_2}$ has the same distribution as $\vec\vartheta_o$ given $\{|\vT|<\infty\}$. As $\MU_o = \psi(\vec \vartheta_o)$, depending on the values of $\vs_2'$, $\efin = -\vs_2 \log \MU_{o_2}$ or  $\efin = -\vs_2 \log(1 -\MU_{o_2})$. \Cref{prop:marginals_of_trees} thus yields that $\efin$ given $\{ \vd_o=2, |\vT_1|=\infty, |\vT_2| < \infty\}$ is a discrete random variable with countable pure point support that is dense in $\RR$. Specifically, we have that
 \begin{align*}
    \Pr\bc{\efin\in \brk{\frac{3a+b}{4}-c,\frac{3b+a}{4}-c} \big\vert \vd_o=2, |\vT_1|=\infty, |\vT_2| < \infty }>0.
\end{align*}
Consequently,  abbreviating $\cE:=\{\vd_o=2, |\vT_1|=\infty, |\vT_2| < \infty \}$, by the independence of $\efin$  and $\einf$ given $\cE$,
\begin{align*}
    \Pr\bc{\vec\vartheta_o\in [a,b]\vert \cE}
    \geq   
    \Pr\bc{\einf\in \brk{c - \frac{b-a}{4},c+\frac{b-a}{4}} \big \vert \cE }
      \cdot \Pr\bc{\efin\in \brk{\frac{3a+b}{4}-c,\frac{3b+a}{4}-c} \big\vert \cE }>0.
\end{align*}
It follows that the support of the distribution of $\vec\vartheta_o$ given $\{|\vT|=\infty\}$ is $\RR$. The desired result then follows from $\MU_o = \psi(\vec \vartheta_o)$ in combination with the continuity and surjectivity of $\psi$.
\end{proof}

\subsection{Proof of Proposition \ref{lem_atoms2}}\label{Sec:Proof_atoms2}
Let $x \in (0,1)$. According to Proposition \ref{Prop_asconvergence} iii), we have
\begin{align}\label{eq:atom_decomp}
    \pi_d(\{x\}) = \Pr\bc{\MU_o=x} = \pr\bc{\MU_o=x \vert |\vT| = \infty} \pr\bc{|\vT|= \infty} + \pr\bc{\MU_o = x \vert |\vT|<\infty}\pr\bc{|\vT|<\infty},
\end{align}
where we set $ \pr\bc{\MU_o=x \vert |\vT| = \infty} \pr\bc{|\vT|= \infty} = 0$ for $d \in (0,1]$. 
However, Lemma \ref{lem:inf_trees_noatoms} implies that in fact for any $d \in (0,2)$, \eqref{eq:atom_decomp} reduces to
\begin{align}
    \pi_d(\{x\}) = \pr\bc{\MU_o = x \vert |\vT|<\infty}\pr\bc{|\vT|<\infty}.
\end{align}
Now by Proposition~\ref{prop:marginals_of_trees}, $\pr\bc{\MU_o = x \vert |\vT|<\infty}$ can only be positive for $x \in \mathbb{Q}\cap(0,1)$, from which it follows that $\mathbb{A}_{\mathrm{p.p.}} \subseteq \mathbb{Q} \cap (0,1)$. 
Furthermore, for $d\in (0,1]$, $|\vT|<\infty$ almost surely. Hence, $\pi_d$ is a discrete measure for $d\in (0,1]$.
Finally, for $d \in (1,2)$, the existence of a non-trivial continuous part with support $[0,1]$ is proven in Lemma~\ref{lem:inf_tree_supp}.

\end{document}